\documentclass{article}
\usepackage[utf8]{inputenc}
\usepackage{amsmath}
\usepackage{amsfonts}
\usepackage{amssymb}
\usepackage{graphicx}
\usepackage{t1enc}
\usepackage{mathrsfs}
\usepackage{amsthm}
\usepackage{imakeidx}
\usepackage[colorinlistoftodos]{todonotes}%

\providecommand{\U}[1]{\protect\rule{.1in}{.1in}}
\newtheorem{theorem}{Theorem}[section]

\newtheorem{conjecture}[theorem]{Conjecture}

\newtheorem{question}[theorem]{Question}

\theoremstyle{definition}
\newtheorem{definition}[theorem]{Definition}

\newtheorem{remark}[theorem]{Remark}

\newcommand{\Z}{\mathbb{Z}}

\title{Open questions in descriptive set theory and dynamical systems}
\author{Jérôme Buzzi, Nishant Chandgotia, Matthew Foreman,\\
Su Gao, Felipe García-Ramos, Anton Gorodetski,\\ François Le Maitre, Federico Rodríguez-Hertz, and Marcin Sabok}

\date{}
\begin{document}

\maketitle
\begin{abstract}
  This file is composed of questions that emerged or were of interest during the workshop "Interactions between Descriptive Set Theory and Smooth Dynamics" that took place in Banff, Canada on 2022. 
\end{abstract}
\tableofcontents

\section{Classification}
\subsection{Topological conjugacy}
\subsubsection{Smooth dynamics}
\textit{-This subsection was written by Anton Gorodetski}

In his ICM talk in Stockholm (in 1962), Stephen Smale proposed a program of studying the topological conjugacy
classes of diffeomorphisms of a given manifold, \cite{S63ICM}. The program was motivated by the related problem on qualitative behaviour of solutions of differential equations on a given manifold.

Due to recent developments \cite{FG}, one should expect that the equivalence relation defined by topological conjugacy on the space of smooth diffeomorphisms $\text{Diff}^{\infty}(M)$ of a given closed (compact, without boundary) manifold $M$ is not Borel, and hence no reasonable classification is possible. Nevertheless, the construction that is used in \cite{FG} essentially uses $C^\infty$-surgery, and cannot be directly extended to, say, to the case of real analytic diffeomorphisms. Hence the first question:

\vspace{5pt}

\begin{question} [Gorodetski] Is there a way to provide a classification of real analytic diffeomorphisms $\text{Diff}^{\omega}(M)$ of a given closed manifold $M$, $\text{dim}\,M\ge 2$, by complete invariants? Is the equivalence relation on $\text{Diff}^{\omega}(M)$ given by topological conjugacy Borel?
\end{question}
\vspace{5pt}

Before we formulate the next question, let us remind that a diffeomorphism of a closed  manifold $M$ is Kupka-Smale if every periodic orbit is hyperbolic, and stable and unstable manifolds of any two periodic orbits are transversal. It is known that the set of Kupka-Smale diffeomorphisms of a given manifold $M$ forms a dense $G_{\delta}$ subspace of $\text{Diff}^r(M)$ for each $r=1, 2, \ldots, \infty$, see \cite{Kup} and \cite{S1963}.

\vspace{5pt}

\begin{question} [Gorodetski] Can the space of Kupka-Smale diffeomorphisms of a given closed manifold be classified up to topological conjugacy by complete invariants? Is topological conjugacy a Borel equivalence relation in the space of Kupka-Smale diffeomorphisms?
\end{question}
\vspace{5pt}

The space of Kupka-Smale diffeomorphisms is a very specific dense $G_\delta$ subset in $\text{Diff}^r(M)$. It is reasonable to ask a more general question:

\vspace{5pt}

\begin{question} [Gorodetski] Is there an open dense subset (or a dense $G_{\delta}$ subset)  in $\text{Diff}^\infty(M)$, $\text{dim}\,M\ge 2$, such that the equivalence relation given by topological conjugacy on the subset is Borel?
\end{question}
 
 \textit{-The next two questions were provided by Matt Foreman and are related to the previous question.} 
 
  \begin{question} [Foreman]
      What is the complexity of the relation of topological conjugacy of Axiom A systems?
  \end{question}

 \begin{question} [Foreman] Let $\mathcal C$ be the collection of diffeomorphisms of a compact manifold $M$ that are topologically conjugate to a structurally stable diffeomorphism.  Let $E$ be the relation of topological conjugacy restricted to $\mathcal C$.  Is $E$ Borel reducible to $\approx$ (equality relation)? 
  \end{question}

\subsubsection{Cantor minimal systems}
\textit{This subsection was written by Su Gao.} 
\begin{question} [Gao] Is the topological conjugacy relation for all Cantor minimal systems Borel?
\end{question}

Kaya \cite{Kaya} proved that the topological conjugacy relation for all \textit{pointed} Cantor minimal systems is a Borel equivalence relation. In fact he showed that it is Borel bireducible to $=^+$. But the answer to the above general problem seems to be unknown.

\begin{question} [Gao] Is the topological flip conjugacy relation for all Cantor minimal systems Borel?
\end{question}

As far as I can tell the answer to this problem is unknown. The seminal work of Giordano, Putnam and Skau \cite{GPS1} \cite{GPS2} started a long string of works toward studying the topological orbit equivalence and the flip conjugacy relations using algebraic methods. Among the most striking results are those of Bezuglyi--Medynets \cite{BM}, Juschenko--Monod \cite{JM}, and Matui \cite{Matui}, which says that the flip conjugacy relation is classified by the topological full groups of these systems, or alternatively, by the commutator subgroups of the topological full groups, which turn out to be infinite simple amenable groups. However, it is not known if the class of countably infinite simple amenable groups is Borel complete.

\begin{question} [Gao] Is the topological conjugacy relation for all minimal subshifts hyperfinite or a universal countable Borel equivalence relation?
\end{question}

By Clemens' result \cite{Clemens} the topological conjugacy relation for all subshifts is a universal countable Borel equivalence relation. It is well established (e.g. \cite{GH} \cite{Thomas}) that $E_0$ is Borel reducible to the topological conjugacy relation for all minimal subshifts. Thus we have $E_0$ as an lower bound and $E_\infty$ as an upper bound for this equivalence relation in the Borel reducibility hierarchy. Its exact complexity is unknown.

\subsubsection{Shifts of finite type}
Using shift equivalence of matrices one can characterize when two SFTs are conjugate \cite{williams73}. This does not give a (finite time) algorithm. 
\begin{question}
    [Foreman]
     Is there an algorithm for determining whether two subshifts of finite type are conjugate by homeomorphisms?
\end{question}
\subsection{Borel conjugacy}
\textit{-This subsection was written by Jérôme Buzzi.} 

My understanding of many results in dynamics is that hard/impossible classification problems come from hyperbolic dynamics, more precisely from points that have a zero Lyapunov exponent or are irregular from the point of view of ergodic theory (the null sets in Birkhoff theorem and Oseledets theorem). 
Note that this irregular subset can carry lots of dynamical behaviors, which in some case can be made invisible up to Borel isomorphism. Indeed, the hyperbolic part, even nonuniform, can be reduced to Bernoulli where Ornstein theory applies. This reduction relies on coding (see in particular \cite{Sarig}) and Borel generator theory (see especially \cite{Hochman}).

In the following I list some results and natural questions following from previous works of several people. 

My goal is to get feedback about these statements:
 \begin{enumerate}
 \item Are those nice classifications restricted to suitably relaxed interesting?
 \item Is the Borelness of the relevant invariants as obvious as I believe?
 \item Is it possible to use them to state and prove {\bf relatively Borel classifications} for stronger notions of isomorphism?
\end{enumerate}

\begin{definition}
Two diffeomorphisms $f,g\in\operatorname{Diff}(M)$ are \emph{Borel conjugate}, written $f\sim_B g$ if there is a Borel automorphism $\psi:M\to M$ st $g\circ\psi=\psi\circ f$.
\end{definition}

The following is a straightforward consequence of the work with Mike Boyle \cite{BB} (up to my correctly understanding what I say):

\begin{theorem}
Let $\mathscr A$ be the set of $C^\infty$ diffeomorphisms on some closed surface $M$ which are topologically mixing and have positive topological entropy.

$(\mathscr A,\sim_B)$ is Borel
\end{theorem}

\begin{proof}
From \cite{BCS}, for all $f,g\in\mathscr A$, we have $f\sim_B g\iff h_\top(f)=h_\top(g)$ and $|Fix(f^n)|=|Fix(g^n)|$ for all $n$. The mechanism here is that the "bad" dynamics (nonhyperbolic and irregular) can be pushed into the "good" one that has positive entropy and no period. These are all Borel invariants.
\end{proof}

\begin{remark}
I believe that the same is true in the conservative world, using metric entropy instead of topological entropy (see \cite{HHTU}).
\end{remark}

In general, it does not seem possible to ignore the "bad dynamics" (eg, it can have maximal entropy or a smaller period than the "good one"). Hence the following definition.

\begin{definition}
A Borel subset $B$ is almost full if $\mu(B)=1$ for all (invariant Borel probability) measure which are hyperbolic (ie, no zero Lyapunov exponent).

Two diffeomorphisms $f,g:M\to M$ are \emph{almost Borel conjugate} if there are almost full Borel subsets $X\subset M$, $Y\subset M$ and a Borel isomorphism $\psi:X\to Y$ which is a conjugacy: $g\circ\psi=\psi\circ g$. We write $f\sim_{aB} g$.
\end{definition}

For surface diffeomorphisms, By Ruelle's inequality, if an ergodic measure has a zero exponent, then it has zero entropy (This is false in general in higher dimension).

\textbf{Almost theorem}
\textit{Let $\mathscr B$ be the set of $C^\infty$ surface diffeomorphisms.}

Then $(\mathscr B,\sim_{aB})$ is Borel.

\begin{proof}[Proof (sort of)]
By a version of Katok's horseshoe theorem, every hyperbolic measure belongs to a homoclinic class.
 
From \cite{BCS}, each homoclinic class has exactly one measure of maximal entropy with the same period. The almost Borel conjugacy class can be obtained by listing the topological entropies and periods of all the homoclinic classes.

I believe that one can do it by determining the set of horseshoes, their entropies and periods and how they nest.
\end{proof}

\begin{remark}
Using \cite{Burguet} one should be able to classify dissipative $C^\infty$ surface diffeomorphisms up to \emph{volume preserving} almost Borel conjugacy.
\end{remark}

\textbf{Almost theorem}
\textit{Let $\mathscr C$ be the set of $C^2$ diffeomorphisms of a closed manifold.
Then $(\mathscr C,\sim_{aB})$ is Borel.}

\begin{proof}[Proof (sort of)]
The same as before, except that one must detect which homoclinic classes have a measure of maximal entropy (and these classes are slightly more abstract, namely the measurable homoclinic classes of \cite{BCS}). I believe that these existence properties are still Borel using the Borel splitting of the manifold according to periods built in \cite{BB}.
\end{proof}

\begin{question} [Buzzi]
I conjecture that a similar statement also holds when discarding only sets of zero Lebesgue measure using metric entropy instead of topological entropy (see \cite{HHTU}), at least in dimension 2 where we have the recent result \cite{Burguet}.
\end{question}

\subsection{Measure isomorphism}
A $K$-automorphism is a measure-preserving dynamical system where every non-trivial factor has positive entropy. Every Bernoulli process is a $K$-automoprhims but the opposite is not true. 
 \begin{question} [Newberger]
  Is measure isomorphism restricted to $K$-automorphisms a Borel equivalence relation? Can it be reduced to an $S_\infty$-action?
 \end{question}\label{OP7}

\begin{question} Is the measure isomorphism relation on ergodic measure-preserving transformations of the unit interval bi-reducible with the maximal Polish group action?
\end{question}

\subsection{Kakutani equivalence}

It is known that the Kakutani equivalence relation is a complete analytic set, and hence not Borel \cite{Gerber}. 

\begin{question} Where does the Kakutani equivalence relation on ergodic measure-preserving transformations sit among the analytic equivalence relations? Is it reducible to an $S_\infty$-action? 
\end{question}
\section{Structure}

\subsection{Completely positive entropy}
\textit{This subsection was written by Felipe García-Ramos.}

Given a manifold $X$ we denote with $C^k(X,X)$ the set of all $C^k$ functions from X into X endowed with the uniform topology. 

We say $f\in C^k(X,X)$ has \textbf{completely positive entropy} if every non-trivial factor has positive topological entropy. We denote the set of such functions with CPE$^k(X)$. 

The space CPE$^0(X)$ contains every system with specification and any topological model or support of a K-system. 
\begin{theorem}
[Darji and García-Ramos \cite{Darji}] For every $n$, $\mathrm{CPE}^0(\Pi^n)$ is complete coanalytic (hence not Borel). 
\end{theorem}
The prove also holds for Lipschitz continuous functions on other compact orientable manifolds but we do not know what happens for smooth functions. 
\begin{question}[García-Ramos]
Let $k\geq 1$ and $n\geq 1$. What is the complexity of $\mathrm{CPE}^k(\Pi^n)$? 
\end{question}

Systems with completely positive entropy may not be mixing. Actually once you add mixing the picture changes. 

The following result is a corollary of a result of Blokh and the fact that mixing maps (denoted by $\mathrm{Mix}$) form a Borel family (see \cite{Darji}).  
\begin{theorem}
$\mathrm{CPE}(\Pi)\cap\mathrm{Mix}$ is Borel (where $\Pi$ is the unit circle).
\end{theorem}

We can contrast this results with the following. 
 
 \begin{theorem}
[Darji and García-Ramos \cite{Darji}] Let $X$ be a Cantor space. We have that $\mathrm{CPE}^k(X)\cap\mathrm{Mix}$ is not Borel. 
\end{theorem}

\begin{question}[García-Ramos]
Let $k\geq 0$ and $n\geq 2$.
Is $\mathrm{CPE}^k(\Pi^n)\cap\mathrm{Mix}$ Borel? 
\end{question}

When $k=\infty$ and $n=2$, the answer could be yes. Buzzi, Crovisier and Sarig \cite{BCS} proved that mixing maps on $\Pi^2$ have a unique measure of maximal entropy which is Bernoulli (and hence has the K-property). If the support was full it would imply that CPE$^{\infty}(\Pi^2)\cap \mathrm{Mix}$ are simply the mixing maps on $\Pi^2$, and hence the set would be Borel. Nonetheless, it is not clear to me if the support of the measure of maximal entropy in full. 

\begin{question}
[García-Ramos]
 Is the measure of maximal entropy fully supported for every smooth mixing map on the torus ($f\in C^{\infty}(\Pi^2,\Pi^2)$)?
\end{question}

\subsection{Kakutani classes}
\begin{question}
[Rodriguez Hertz] Are Kakutani equivalence classes Borel?   If so what is the complexity of the isomorphism relation on that class?
\end{question}
The Kakutani class of the odometer are the zero entropy loosely Bernoulli systems. This seems to be Borel using the $\overline{f}$ characterization of Feldman and Katok.

\section{Group actions}
\subsection{Free group}
\begin{question}
[Foreman]
     Let $X$ be the space of continuous actions of the free group $F_2$ on two generators. Is the collection of Bernoulli actions measure isomorphic to Bernoulli $F_2$-actions a Borel set? (Borel actions?)
\end{question}
 \subsection{Conjugacy for ergodic $\Z^2$ actions}
\textit{This subsection was written by Nishant Chandgotia. }

	How complicated is the conjugacy relation for $\Z^2$ actions as compared to $\Z$ actions?

Let me try to give this vague question some direction. We concentrate henceforth on the $\Z^2$ shift-action on the free part of the full shift $\{0,1\}^{\Z^2}$ which we will denote by $X_2$. The shift action will be denoted by $\sigma$. All that I am about to say generalises to free $\Z^d$ actions on Polish spaces but for simplicity we will stick to the stated setting. 

Think of $\Z^2$ both as a group and as a directed Cayley graph with standard generators. A directed bi-infinite Hamiltonian path in $\Z^2$ is a connected subgraph with the same set of vertices as $\Z^2$ such that each vertex has exactly one incoming and one outgoing edge. There is a natural $\Z^2$ action on these paths by translation. Let $X_H$ denote the space of all directed bi-infinite Hamiltonian paths.

I heard about the following result from Brandon Seward. 

\begin{theorem}\cite{gaotoappear} There is a Borel equivariant map $\Phi$ from $(X_2, \sigma)$ to $(X_H, \sigma)$.
	\end{theorem}

One of the consequences of this result is that it gives a (very special) orbit equivalence from $(X_2, \sigma)$ to a $\Z$ Polish dynamical system which we denote hence forth by $(X_2, T_\Phi)$. In addition one has that if $\mu$ is an invariant ergodic probability measure for $(X_2, \sigma)$ then it is also an invariant ergodic probability measure for $(X_2, T_\Phi)$. Let $\mathcal M_e$ denote the space of invariant probability measures for $(X_2,\sigma)$. Here is a rather (wild) question.

\begin{question}
[Chandgotia]
Can $\Phi$ be chosen such that for all $\mu, \nu\in M_e$, $(X_2, \mu, \sigma)$ is conjugate to $(X_2, \nu, \sigma)$ if and only if $(X_2, \mu, T_\Phi)$ is conjugate to $(X_2, \nu, T_\Phi)$?
\end{question}

\section{Polish groups}

\subsection{Orbit equivalence relations}
\textit{This subsection was written by Marcin Sabok.}

Given a  standard probability space $(X,\mu)$, the group Aut$(X,\mu)$ acts on itself by conjugation and induces an orbit equivalence relation. Foreman, Rudolph and Weiss \cite{foremanconjugacyproblemergodic2011} showed that  the set of ergodic elements in Aut$(X,\mu)$ that are isomorphic to their inverse is a complete analytic set, which implies that the orbit equivalence relation is a complete analytic subset of Aut$(X,\mu)\times$Aut$(X,\mu)$. On the other hand, in \cite{FW} Foreman and Weiss showed that the action of the group Aut$(X,\mu)$ on the ergodic elements of Aut$(X,\mu)$ by conjugation is turbulent in the sense of Hjorth.

In general, given a Borel action of a Polish group on a standard Borel space we consider the induced orbit equivalence relation. We say that an orbit equivalence relation E is complete if every orbit equivalence relation is Borel reducible to E.

\begin{conjecture} [Sabok] The action of $\mathrm{Aut}(X,\mu)$ on itself by conjugation is a complete orbit equivalence relation.
\end{conjecture}

In \cite{Sab} it is shown that the affine homeomorphism of Choquet simplices is a complete orbit equivalence relation. Recall that by the result of Downarowicz \cite{Dow} every Choquet simplex is realized as the simplex of invariant measures on a (Toeplitz) subshift, however this construction is not a Borel reduction of affine homeomorphism of Choquet simplices to the conjugacy of subshifts. In fact such a reduction cannot exist since the topological conjugacy of subshifts is a countable Borel equivalence relation, but it seems plausible that changing subshifts to more complicated systems may lead to a reduction of the affine homeomorphism of Choquet simplices to the conjugacy of those systems in the sense of measurable dynamics.

The above conjecture is connected to the question on which groups can induce complete orbit equivalence relations. Recall that a Polish group $G$ is universal if every Polish group can be embedded into $G$. It is not known whether a group that is not a universal Polish group can induce a complete orbit equivalence relation. Note that the group $U(H)$ of unitary operators on the separable Hilbert space $H$ is not a universal Polish group since there are examples of Polish group (e.g., the so-called exotic group) that do not admit any nontrivial representations on a Hilbert space. However, Aut$(X,\mu)$ embeds into $U(H)$ via the Koopman representation. Thus, the above conjecture implies a positive answer to \cite[Question 9.3]{Sab} on whether the group $U(H)$ can induce a complete orbit equivalence relation

\subsection{Full groups}
\textit{The following two subsections were written by François Le Maitre.}

Dye's theorem states that up to conjugacy, any two ergodic measure-preserving transformations on a standard probability space $(X,\mu)$ share the same orbits. 
This can be restated as follows. Let us first fix an ergodic transformation $T_0$, and define its full group $[T_0]$ as the group of bimeasurable bijections $U:X\to X$ such that for all $x$, the point $U(x)$ belongs to the $T_0$-orbit of $x$. Then every ergodic measure-preserving transformation is conjugate to an element of $[T_0]$ whose orbits are equal to those of $T_0$.
This motivates the following question of Thouvenot, which in a nutshell asks whether Ornstein's beautiful machinery can work inside $[T_0]$.

\begin{question}[Thouvenot]
	Let $T_1,T_2$ be two Bernoulli shifts with equal entropy, suppose that they have the same orbits as $T_0$. Is there $T\in[T_0]$ such that $T_1=TT_2T^{-1}$?
\end{question}

More generally, one can wonder what is the relationship between the conjugacy relation on the group Aut$(X,\mu)$ of all measure-preserving transformations and the conjugacy relation on $[T_0]$. 

\begin{question}\label{qu: red Aut to full}
	Does the conjugacy relation on $\mathrm{Aut}(X,\mu)$ Borel reduce to the conjugacy relation on $[T_0]$? Does it when we restrict to ergodic transformations?
\end{question}

This question is particularly relevant towards answering in the negative Sabok's conjecture of the universality of the conjugacy relation on $\mathrm{Aut}(X,\mu)$. Indeed, if its answer is positive then since $[T_0]$ admits a bi-invariant compatible metric (given by $d_u(S,T)=\mu(\{x\in X\colon S(x)\neq T(x)\})$, see e.g.~\cite[Prop.~3.2]{kechrisGlobalaspectsergodic2010}), the equivalence relations it induces cannot be universal for Polish group actions by work of Allison and Panagiotopoulos \cite{allisonDynamicalObstructionsClassification2021}. Actually, their work also provides a dynamical analogue of turbulence called \emph{unbalancedness} which rules out Borel reduction to equivalence relations coming from actions of Polish groups admitting a compatible biinvariant metric, so a positive answer to the following question implies a negative answer to Question \ref{qu: red Aut to full}.

\begin{question}
	Is the conjugacy action of $\mathrm{Aut}(X,\mu)$ on itself unbalanced? If yes, is it still if we restrict to ergodic transformations?
\end{question}

A positive answer to the following  question would also yield a negative answer to Question \ref{qu: red Aut to full} thanks to \cite{foremanconjugacyproblemergodic2011}.

\begin{question}
	Is the conjugacy relation on $[T_0]$ Borel? If no, is it when restricted to ergodic transformations?
\end{question}
For other properties of the conjugacy relation on $[T_0]$, especially turbulence, see \cite[Chap.~5]{kechrisGlobalaspectsergodic2010}.

\subsection{Topological full groups}
Topological full group of minimal countable discrete group actions on Cantor spaces and their subgroups have proven to be an invaluable source of examples of countable groups \cite{Matui,JM,nekrashevychSimplegroupsdynamical2017,nekrashevychPalindromicsubshiftssimple2018}.
They are defined as follows: given a countable group $\Gamma$ and a continuous action $\Gamma \curvearrowright X$ on the Cantor space $X=\{0,1\}^\mathbb{N}$, the \emph{topological full group} of the action is the group of all homeomorphisms $h$ of the Cantor space such that there is a clopen partition $(U_1,\dots,U_n)$ of $\{0,1\}^\mathbb{N}$ and $\gamma_1,\dots,\gamma_n\in\Gamma$ such that 
$$ \text{for all }x\in U_i, h(x)=\gamma_i\cdot x.$$
Equivalently, $h$ is in the topological full group of the action if and only if there is a map $c_h: X\to \Gamma$ such that $c_h$ is \emph{continuous} and $h(x)=c_h(x)\cdot x$ for all $x\in X$. This map $c_h$ is called the \textbf{cocycle} of $h$, the equation $h(x)=c_h(x)\cdot x$ uniquely defines it when the $\Gamma$-action was free. It is important to note here that the continuity of $c_h$ is a fundamental hypothesis to make: without it, we obtain an uncountable group which can even fail to have a Polish group topology or to be Borel \cite{ibarluciaFullGroupsMinimal2016}. 

It has been noted that as soon as we replace $X$ by a connected space, then the topological full group is equal to the acting group since the acting group is countable discrete. But what if the acting group is not discrete, e.g. the real line? It turns out this question has been studied by 
Matte-Bon and Triestino for suspension flows over Cantor minimal systems; they showed that one obtains a Polish group\footnote{ As they explain in their introduction, this Polish group was first defined and studied as the group of homeomorphisms of the suspension space which are isotopic to the identity, but turns out to be equal to the group of homeomorphisms which preserve the orbits of the flow and have a continuous cocycle.} which moreover arises as a natural completion of analogues of Thompson's group $T$ which they build \cite[Appendix A]{mattebonGroupsPiecewiseLinear2020}.

More generally, one can use the approach from \cite{carderiMorePolishfull2016} to show that given a Polish group $G$ acting continuously on a compact Polish space $X$, one can associate to it a Polish group, the topological full group of the action. When the action is free, this is the group of homeomorphisms which preserve the $G$-orbits and whose cocycle is continuous. In the non-free case, this is the group of homeomorphisms which preserve the $G$ orbits and which \emph{admit} a continuous cocycle. We briefly sketch the construction, assuming some familiarity with \cite[Sec.~3.2]{carderiMorePolishfull2016}.

\begin{proof}[Sketch of the construction]
	The group $\mathcal C(X,G)$ of continuous maps from $X$ to $G$ is a Polish group for the compact-open topology. The set $\mathcal C(X,X)$ of continuous maps from $X$ to $X$ is Polish also for the compact-open topology. $ \mathrm{Homeo}(X)$ is a $G_\delta$ in there because it is Polish for the compact-open topology. It acts continuously on $\mathcal C(X,G)$ by precomposition.
	
	We have a natural map $\Phi:\mathcal C(X,G)\to \mathcal C(X,X)$ given by $f\mapsto (x\mapsto f(x)\cdot x)$. Such a map is continuous, so the preimage of $ \mathrm{Homeo}(X)$ is $G_\delta$ in $\mathcal C(X,G)$. Let us call 
	$$\widetilde{[G]}_c:=\Phi^{-1}( \mathrm{Homeo}(X))$$
	the \textbf{pre-full group} of the action (it is just the space of possible continuous cocycles of the homeomorphisms which preserve the $G$-orbits and admit such a continuous cocycle). Being a $G_\delta$ subset of a Polish space, it is Polish.

	Now  consider the group law on $\widetilde{[G]}_c$ which comes from composition of the corresponding homeomorphisms, namely 
	$$f\cdot g(x)=f[\Phi(g)(x)]g(x)=f(g(x)\cdot x)g(x).$$
	Since $\Phi$ is continuous and $ \mathrm{Homeo}(X)$ acts continuously on $\mathcal C(X,G)$, this multiplication is continuous and we conclude that the pre-full group is a Polish group. Now the \textbf{topological full group} of the action is the quotient of the pre-full group by the closed normal full group consisting of those elements $f$ such that $\Phi(f)=\mathrm{id}_X$. It is thus a Polish group, whose topology clearly refines the compact-open topology on $ \mathrm{Homeo}(X)$.
\end{proof}

\begin{question}
	What can we say about these topological full groups, especially when the acting group is the real line? For instance, when are they coarsely bounded in the sense of \cite{rosendalTopologicalVersionBergman2009}? When do they admit a dense finitely generated group\footnote{See \cite[Thm. A]{mattebonGroupsPiecewiseLinear2020} for a partial answer for suspension flows over the Cantor space.}?  How many such groups are there\footnote{See \cite[Sec.~10]{mattebonGroupsPiecewiseLinear2020} for a complete answer for suspension flows over the Cantor space.}? Can one classify their actions by homeomorphisms?
\end{question}

Note that when we have an action with more structure, e.g. the flow associated to a vector field, it would also be interesting to study full groups with more stringent regularity conditions, e.g.~where the elements are say $\mathcal C^1$-diffeomorphisms with $\mathcal C^1$ cocycles.

\section{Realization and universality}

\subsection{Smooth realization}
\textit{This subsection was extracted from questions of Federico Rodriguez Hertz on the open problem sessions of the previously mentioned conference}

Let $(I,\lambda)$ be the interval Lebesgue space and $(\Pi^n,\lambda_n)$ the $n$-torus Lebesgue space. 
We have that $\text{Aut}_{h<\infty}(I,\lambda)$ is the space of (Lebesgue) measure-preserving automorphisms of $I$ with finite entropy and $\text{Diff}(\Pi^n,\lambda_n)$ the space of (Lebesgue) measure-preserving diffeomorphisms of $\Pi^n$. 
\begin{question}
[Rodriguez Hertz] Is there a Borel map $G:\text{Aut}_{h<\infty}(I,\lambda)\rightarrow \bigcup_n \text{Diff}(\Pi^n,\lambda_n)$, such that $(I,\lambda,T)$ is isomorphic to $(\Pi^n,\lambda_n, G(T))$. 
\end{question}

There seems to be no indication that this result is true, so probably the classes (of automorphisms) need to be restricted, for example we could consider only the K-automorphisms. 

It is still open if any (non-periodic) odometer can be smoothly realized. 

\begin{question}
[Rodriguez Hertz] Is there a Borel map $G:\text{Aut}_{h<\infty}(I,\lambda)\rightarrow \bigcup_n \text{Diff}(\Pi^n,\lambda_n)$, such that $(I,\lambda,T)$ is Kakutani equivalent to $(\Pi^n,\lambda_n, G(T))$. 
\end{question}

Again this seems to be a complicated question, the only advantage is that on this situation odometers can be realized with rotations on the torus.  
\begin{question}
[Rodriguez Hertz] If a measure-preserving system is Kakutani realizable can we conclude it is also isomorphically realizable? 
\end{question}

\subsection{Uniquely ergodic actions}

Krieger's generator theorem allows you to realize finite entropy measure-preserving systems with subshifts.
A finer construction (the Jewett-Krieger theorem) allows you to realize measure-preserving transformations with uniquely ergodic topological dynamical systems. 

Krieger's generator type theorem is known for general group actions. 

\begin{question}
[Seward] Do we have uniquely ergodic topological models for every measure-preserving group action of a free group?
\end{question}
\subsection{Universality}
Universality is similar to realization but here we have more flexibility in taking the measure on the image of $G$, that is the measure which is being used to be realized is taken to be fully supported but not necessarily the Lebesgue measure. In this sense it is known that smooth diffeomorphisms of the 2-torus (with SL$(2,\Z)$ matrices) are universal for the finite entropy measure preserving transformations. This argument is a refinement of Krieger's generator theorem. Another positive result was obtained by Soo-Quas. 

\begin{question}
    [Rodriguez Hertz] Which families of systems are universal? 
\end{question}

The next subsection is also related to universality. 

\subsection{Embeddings into systems with non-uniform specification}
\textit{This subsection is written by Nishant Chandgotia. }
\begin{question}\label{question: Main}\cite{chandgotia2022borel}
	Let $(\mathbb T^d, R)$ be a toral automorphism which is ergodic for the Lebesgue measure. Let $(Y, S)$ be a homeomorphism of a Polish space without any invariant probability measure. Is there an equivariant Borel embedding of $(Y, S)$ into $(\mathbb T^d, R)$?
\end{question}

Hochman \cite{MR3880210} proved that such an embedding is possible where the toral automorphism $(\mathbb T^d, R)$ is replaced by the full shift $(\{0,1\}^{\Z}, \sigma )$ (or more generally a mixing shift of finite type). Using this, I believe it is immediate that the answer to Question \ref{question: Main} is yes when $R$ is hyperbolic. The challenge comes from the lack of symbolic coding in the non-hyperbolic case \cite{lindenstrauss2004invariant, lindenstrauss2005symbolic}. 

Moving out of the compressible setting, in \cite{MR3453367} Quas and Soo prove that for any ``appropriate'' Polish actions $(Y, S)$, there exists an equivariant Borel embedding of $(Y,S)$ into $(\mathbb T^d, R)$ modulo a $\mu$-null set where $\mu$ is an invariant probability measure. To prove this they critically used the fact that ergodic toral automorphisms satisfy a nice mixing property called non-uniform specification.

  Let $X$ be a compact metric space. We say that a $\Z$-action, $(X,T)$ satisfies \emph{non-uniform specification} if 
 there exists a sequence of increasing functions $g_n:(0,1) \to (0,\infty)$ satisfying the following conditions:
 \begin{itemize}
 	\item For every $\epsilon>0$, $\lim_{n \to \infty}g_n(\epsilon)= 0$,
 	\item For  every $n_1,\ldots,n_s \in \mathbb{N}$ , $i_1,\ldots,i_s \in \mathbb{Z}$ and  $\epsilon>0$ such that 
 	$$\{i_1 + (1+g_{n_1}(\epsilon))[1, n_1],\ldots, i_s + (1+g_{n_s}(\epsilon))[1, n_s]\}$$
 	are pairwise disjoint, 
 	and any $x_1,\ldots,x_s \in X$ there exists $x \in X$ such that $d_X(T^{i_j+t}(x), T^{i_j+t}(x_j)) < \epsilon$ for all $1\leq t\leq n_j$ and $1 \le j \le s$.
 \end{itemize}
 
 Quas and Soo asked whether their result can be extended to topological dynamical systems with non-uniform specification. This was proved affirmatively by Burguet \cite{burguet2020topological}. Burguet proved more than what Quas and Soo were asking for: He showed that under suitable assumptions on the entropy any free Polish $\Z$-action $(Y, S)$ can be embedded into a system with non-uniform specification modulo a universally null set (that is a set with measure zero for all invariant probability measures). Naturally the question arises whether embedding can be extended to the full space.
 
 We remark that we do not even know the answer to the following question. 
 
 \begin{question}\label{question: Not_Main}\cite{chandgotia2022borel}
 	Let $(\mathbb T^d, R)$ be a toral automorphism which is ergodic for the Lebesgue measure. Let $(Y, S)$ be a homeomorphism of a Polish space without any invariant probability measure. Is there an equivariant Borel map of $(Y, S)$ into the free part of $(\mathbb T^d, R)$?
 \end{question}
 
Along with Tom Meyerovitch we gave an alternative proof of David Burguet's result in \cite{https://doi.org/10.1112/plms.12398} introducing a specification like property called flexibility. We used an approximation method to construct our embeddings, where at each stage of our construction we modify the approximate embedding on a small part of the space, and finally use Borel-Cantelli to ensure that the part of the space where our approximation doesn't converge has measure zero for any invariant probability measure. Such methods can't be used for constructing embeddings of Polish actions because there is no alternative to Borel-Cantelli lemma in this setting. 

\section{Ackowledgements}

We thank the BIRS administration for providing the necessary support during the workshop "Interactions between Descriptive Set Theory and Smooth Dynamics".


\begin{thebibliography}{99}




 
	\bibitem{allisonDynamicalObstructionsClassification2021}
	S. Allison and A. Panagiotopoulos.
	\newblock Dynamical obstructions to classification by (co)homology and other
	{{TSI-group}} invariants.
	\newblock {\em Transactions of the American Mathematical Society},
	374(12):8793--8811, 2021.
	
\bibitem{BO}
S. Benovadia. Symbolic dynamics for non-uniformly hyperbolic diffeomorphisms of compact smooth manifolds. {\em Journal of Modernd Dynamics 13:43-113,} 2018.

\bibitem{BM}
S. I. Bezuglyi and K. Medynets. Full groups, flip conjugacy, and orbit equivalence of Cantor minimal systems, {\em Colloquium Mathematicum} 110:409--420, 2006.

\bibitem{BB}
M. Boyle and J. Buzzi The almost Borel structure of surface diffeomorphisms, Markov shifts and their factors.{\em Journal of the European Mathemtical Society 19(9):2739–2782,} 2017.

\bibitem{burguet2020topological}
D. Burguet.
\newblock Topological and almost {B}orel universality for systems with the weak
  specification property.
\newblock {\em Ergodic Theory and Dynamical Systems}, 40(8):2098--2115, 2020.

\bibitem{Burguet}
D. Burguet. SRB measures for $C^\infty$ surface diffeomorphisms. {\em ArXiv:2111.0665}

\bibitem{BCS}
J. Buzzi, S. Crovisier, and O, Sarig. Measures of maximal entropy for surface diffeomorphisms. {\em Annals of Mathematics 195(2):421-508,} 2022.

\bibitem{carderiMorePolishfull2016}
	A. Carderi and F. Le~Ma{\^i}tre.
	\newblock More {{Polish}} full groups.
	\newblock {\em Topology and its Applications}, 202:80--105, 2016.
	
\bibitem{https://doi.org/10.1112/plms.12398}
N. Chandgotia and T. Meyerovitch.
\newblock Borel subsystems and ergodic universality for compact $\mathbb
  {Z}^d$-systems via specification and beyond.
\newblock {\em Proceedings of the London Mathematical Society},
  123(3):231--312, 2021.

\bibitem{chandgotia2022borel}
N. Chandgotia and S. Unger.
\newblock Borel factors and embeddings of systems in subshifts.
\newblock {\em  arXiv:2203.09359}.

\bibitem{Clemens}
J. Clemens. Isomorphism of subshifts is a universal countable Borel equivalence relation, {\em Israel Journal of Mathematics} 170:113--123, 2009.

\bibitem{Darji} U. Darji and F. García-Ramos. Local entropy theory and descriptive complexity. {\em arXiv:2107.09263}.

\bibitem{Dow} T. Downarowicz. 
\newblock The Choquet simplex of invariant measures for minimal flows. 
\newblock {\em Israel Journal of Mathematics},
	74:241--256, 1999.
 
\bibitem{foremanconjugacyproblemergodic2011}
	Matthew Foreman, Daniel Rudolph, and Benjamin Weiss.
	\newblock The conjugacy problem in ergodic theory.
	\newblock {\em Annals of Mathematics}, 173(3):1529--1586, 2011.

\bibitem{FG} M.\,Foreman and A.\,Gorodetski. Work in progress. 	
\bibitem{FW} M. Foreman and B. Weiss.
\newblock An anti-classification theorem for ergodic measure preserving transformations. 
\newblock {\em Journal of European Mathematical Society},
	6.3:277--292, 2004.

\bibitem{GH}
S. Gao and A. Hill. Topological isomorphism for rank-1 systems, {\em Journal d'Analyse Math\'{e}matique} 128:1-49, 2016.

\bibitem{GPS1}
T. Giordano, I. F. Putnam, and C. F. Skau. Topological orbit equivalence and $C^*$-crossed products, {\em Journal f\"{u}r die Reine und Angewandte Mathematik} 469:51-111, 1995.

\bibitem{GPS2}
T. Giordano, I. F. Putnam, and C. F. Skau. Full groups of Cantor minimal systems, {\em Israel Journal of Mathematics} 111:285--320, 1999.


\bibitem{gaotoappear}
S. Gao, S. Jackson, E. Krohne, and Brandon Seward.
\newblock Borel combinatorics of countable group actions.
\newblock {\em To appear}.

\bibitem{Gerber} M. Gerber and P. Kunde. Anti-classification results for the Kakutani equivalence relation. {\em arXiv:2109.06086}.

\bibitem{MR3880210}
M. Hochman.
\newblock Every {B}orel automorphism without finite invariant measures admits a two-set generator.
\newblock {\em Journal of the European Mathematical Society}, 21(1):271--317, 2019.
	
\bibitem{Hochman}
M. Hochman. Every Borel automorphism without finite invariant measures admits a two-set generator.{\em Journal of the European Mathematical Society 21(1):271–317}, 2018.

\bibitem{ibarluciaFullGroupsMinimal2016}
	T. Ibarluc{\'i}a and J. Melleray.
	\newblock Full groups of minimal homeomorphisms and {{Baire}} category methods.
	\newblock {\em Ergodic Theory and Dynamical Systems}, 36(2):550--573,
	2016.

\bibitem{JM}
K. Juschenko and N. Monod. Cantor systems, piecewise translations and simple amenable groups, {\em Annals of Mathematics} 178:775--787, 2012.

	
	
	\bibitem{kechrisGlobalaspectsergodic2010}
	A. S. Kechris.
	\newblock {\em Global Aspects of Ergodic Group Actions}, volume 160 of {\em
		Mathematical {{Surveys}} and {{Monographs}}}.
	\newblock {American Mathematical Society, Providence, RI}, 2010.

\bibitem{Kaya}
B. Kaya. The complexity of topological conjugacy of pointed Cantor minimal systems, {\em Archive for Mathematical Logic} 56(3):215--235, 2017.

\bibitem{Kup} I.\,Kupka.
Contribution a la theorie des champs generiques,
{\it Contributions to Differential Equations} {\em2} 457--484, 1963.

\bibitem{lindenstrauss2004invariant}
E. Lindenstrauss and K. Schmidt.
\newblock Invariant sets and measures of nonexpansive group automorphisms.
\newblock {\em Israel Journal of Mathematics}, 144(1):29--60, 2004.

\bibitem{lindenstrauss2005symbolic}
E. Lindenstrauss and K. Schmidt.
\newblock Symbolic representations of nonexpansive group automorphisms.
\newblock {\em Israel Journal of Mathematics}, 149(1):227--266, 2005.

\bibitem{Matui}
H. Matui. Some remarks on topological full groups of Cantor minimal systems II, {\em Ergodic Theory and Dynamical Systems} 33 (2013), 1542--1549.

	

	
	\bibitem{mattebonGroupsPiecewiseLinear2020}
	N. Matte~Bon and M. Triestino.
	\newblock Groups of piecewise linear homeomorphisms of flows.
	\newblock {\em Compositio Mathematica}, 156(8):1595--1622, August 2020.
	
	\bibitem{nekrashevychSimplegroupsdynamical2017}
	V.~Nekrashevych.
	\newblock Simple groups of dynamical origin.
	\newblock {\em Ergodic Theory and Dynamical Systems 39(3):1-26}, 2017.
	
	\bibitem{nekrashevychPalindromicsubshiftssimple2018}
	V. Nekrashevych.
	\newblock Palindromic subshifts and simple periodic groups of intermediate
	growth.
	\newblock {\em Annals of Mathematics} 187(3):667--719, 2018.

\bibitem{MR3453367}
A. Quas and T. Soo.
\newblock Ergodic universality of some topological dynamical systems.
\newblock {\em Transactions of the American Mathematical Society}, 368(6):4137--4170, 2016.


\bibitem{HHTU}
F. Rodriguez-Hertz, J. Rodriguez-Hertz, A. Tahzibi, and R. Urès.  Uniqueness of SRB measures for transitive diffeomorphisms of surfaces, {\em Communications in Mathematical Physics 306:35–49}, 2011.

	
	\bibitem{rosendalTopologicalVersionBergman2009}
	C. Rosendal.
	\newblock A topological version of the {{Bergman}} property.
	\newblock {\em Forum Mathematicum}, 21(2):299--332, 2009.

\bibitem{Sab} M. Sabok.
\newblock Completeness of the isomorphism problem for separable C*-algebras. 
\newblock {\em Inventiones Mathematicae},
	204.3:833--868, 2016.
 
\bibitem{Sarig}
O. Sarig. Symbolic dynamics for surface diffeomorphisms with positive entropy. {\em Journal of the American Mathematical Society 26(2):41–426}, 2013.


\bibitem{S63ICM} S. Smale. Dynamical systems and the topological conjugacy problem for diffeomorphisms, {\it 1963 Proceedings of the ICM (Stockholm, 1962)}

\bibitem{S1963}  S. Smale.  Stable manifolds for differential equations and diffeomorphisms, {\it  Annali della Scuola Normale Superiore di Pisa}, {\em17}:97--116, 1963.

\bibitem{Thomas}
S. Thomas. Topological full groups of minimal subshifts and just-infinite groups, in {\em Proceedings of the 12th Asian Logic Conference}, 298--313. World Scientific Publishing, Hackensack, NJ, 2013.




	
	
	


	
	\bibitem{williams73}R.F. Williams. Classification of subshifts of finite type.\textit{ Annals of Mathematics}. 98 (1973), 120-153; erratum." \textit{Annals of Mathematics} 99: 380-381, 1974.
	
\end{thebibliography}
\end{document}